\documentclass{article}
\usepackage{nips10submit_e,times}
\usepackage{amssymb,amsmath,amsthm}
\usepackage{graphicx}

\RequirePackage[dvips]{hyperref}
\RequirePackage{hypernat}

\def\s{\sigma}
\def\ee{\mathbf{e}}
\def\ss{\mathbf{E}}
\def\so{\mathbf{E}_0}
\def\e{\epsilon}
\title{Extended Bayesian Information Criteria for Gaussian Graphical Models}
\def\trace{\mathrm{trace}}

\newtheorem{theorem}{Theorem}
\newtheorem*{mtheorem}{Main Theorem}
\newtheorem{lemma}{Lemma}
\newtheorem{alemma}{Lemma}[section]

\theoremstyle{remark}
\newtheorem{remark}{Remark}

\author{Rina Foygel\\
University of Chicago\\
\texttt{rina@uchicago.edu} \\
\And
Mathias Drton \\
University of Chicago\\
\texttt{drton@uchicago.edu} \\
}
\nipsfinalcopy 

\begin{document}
\maketitle

\begin{abstract}%
  Gaussian graphical models with sparsity in the inverse covariance
  matrix are of significant interest in many modern applications.  For
  the problem of recovering the graphical structure, information
  criteria provide useful optimization objectives for algorithms
  searching through sets of graphs or for selection of tuning
  parameters of other methods such as the graphical lasso, which is a
  likelihood penalization technique.  In this paper we establish the
   consistency of an extended Bayesian information criterion
  for Gaussian graphical models in a scenario where both the number of
  variables $p$ and the sample size $n$ grow.  Compared to earlier
  work on the regression case, our treatment allows for growth in the
  number of non-zero parameters in the true model, which is necessary
  in order to cover connected graphs.  We demonstrate the performance
  of this criterion on simulated data when used in conjunction with the
  graphical lasso, and verify that the criterion indeed performs
  better than either cross-validation or the ordinary Bayesian information
  criterion when $p$ and the number of non-zero parameters $q$ both scale with $n$.
\end{abstract}

\section{Introduction}
\label{section_Intro}

This paper is concerned with the problem of model selection (or
structure learning) in Gaussian graphical modelling.  A Gaussian
graphical model for a random vector $X=(X_1,\dots,X_p)$ is determined
by a graph $G$ on $p$ nodes.  The model comprises all multivariate
normal distributions $N(\mu,\Theta^{-1})$ whose inverse covariance
matrix satisfies that $\Theta_{jk}=0$ when $\{j,k\}$ is not an edge in
$G$.  For background on these models, including a discussion of the
conditional independence interpretation of the graph, we refer the
reader to \cite{lauritzen:1996}.

In many applications, in particular in the analysis of gene expression
data, inference of the graph $G$ is of significant interest.
Information criteria provide an important tool for this problem.  They
provide the objective to be minimized in (heuristic) searches over the
space of graphs and are sometimes used to select tuning parameters in
other methods such as the graphical lasso of \cite{friedman:2008}. In
this work we study an extended Bayesian information criterion (BIC)
for Gaussian graphical models.  Given a sample of $n$ independent and
identically distributed observations, this criterion takes the form
\begin{equation}
  \label{eq:graphical-ebic}
  BIC_{\gamma}(\ss)=-2l_n(\hat{\Theta}(\ss))+|\ss|\log n+4|\ss|\gamma
  \log p,
\end{equation}
where $\ss$ is the edge set of a candidate graph and
$l_n(\hat{\Theta}(\ss))$ denotes the maximized log-likelihood function
of the associated model.  (In this context an edge set comprises unordered pairs
$\{j,k\}$ of distinct elements in $\{1,\dots,p\}$.) The criterion is
indexed by a parameter $\gamma\in[0,1]$; see the Bayesian
interpretation of $\gamma$ given in \cite{chen:2008}.  If $\gamma=0$,
then the classical BIC of \cite{schwarz:1978} is recovered, which is
well known to lead to (asymptotically) consistent model selection in
the setting of fixed number of variables $p$ and growing sample size
$n$.  Consistency is understood to mean selection of the smallest true
graph whose edge set we denote $\so$. Positive $\gamma$ leads to
stronger penalization of large graphs and our main result states that
the (asymptotic) consistency of an exhaustive search over a restricted
model space may then also hold in a scenario where $p$ grows
moderately with $n$ (see the Main Theorem in
Section~\ref{section_Main}).  Our numerical work demonstrates that
positive values of $\gamma$ indeed lead to improved graph inference
when $p$ and $n$ are of comparable size (Section~\ref{section_Sim}).

The choice of the criterion in (\ref{eq:graphical-ebic}) is in analogy to a
similar criterion for regression models that was first proposed in
\cite{bogdan:2004} and theoretically studied in \cite{chen:2008,chen:2010}.
Our theoretical study employs ideas from these latter two papers as well as
distribution theory available for decomposable graphical models.  As
mentioned above, we treat an exhaustive search over a restricted model
space that contains all decomposable models given by an edge set of
cardinality $|\ss|\le q$.  One difference to the regression treatment of
\cite{chen:2008,chen:2010} is that we do not fix the dimension bound $q$
nor the dimension $|\ss_0|$ of the smallest true model.  This is necessary
for connected graphs to be covered by our work.

In practice, an exhaustive search is infeasible even for moderate
values of $p$ and $q$.  Therefore, we must choose some method for
preselecting a smaller set of models, each of which is then scored by
applying the extended BIC (EBIC).  Our simulations show that the
combination of EBIC and graphical lasso gives good results well beyond
the realm of the assumptions made in our theoretical analysis.  This
combination is consistent in settings where both the lasso and the
exhaustive search are consistent but in light of the good theoretical
properties of lasso procedures (see \cite{ravikumar:2008}), studying
this particular combination in itself would be an interesting topic
for future work.

\section{Consistency of the extended BIC for Gaussian graphical models}
\label{section_Main}

\subsection{Notation and definitions}
\label{subsec:prelim}

In the sequel we make no distinction between the edge set $\ss$ of a
graph on $p$ nodes and the associated Gaussian graphical model.
Without loss of generality we assume a zero mean vector for all
distributions in the model.  We also refer to $\ss$ as a set of
entries in a $p\times p$ matrix, meaning the $2|\ss|$ entries indexed
by $(j,k)$ and $(k,j)$ for each $\{j,k\}\in\ss$.  We use $\Delta$ to
denote the index pairs $(j,j)$ for the diagonal entries of the matrix.

Let $\Theta_0$ be a positive definite matrix supported on
$\Delta\cup\so$. In other words, the non-zero entries of $\Theta_0$ are
precisely the diagonal entries as well as the off-diagonal positions
indexed by $\so$; note that a single edge in $\so$ corresponds to two
positions in the matrix due to symmetry.  Suppose the random vectors
$X_1,\dots,X_n$ are independent and distributed identically according
to $N(0,\Theta_0^{-1})$.  Let $S=\frac{1}{n}\sum_i X_iX_i^T$ be the
sample covariance matrix. The Gaussian log-likelihood function
simplifies to 
\begin{equation}
  \label{eq:loglik}
  l_n(\Theta)=\frac{n}{2}\left[\log\det(\Theta)-\trace(S\Theta)\right].
\end{equation}

We introduce some further notation.  First, we define the maximum
variance of the individual nodes:
$$
\s^2_{\max}=\max_j (\Theta_0^{-1})_{jj}.
$$
Next, we define $\theta_0=\min_{\ee\in\so}|(\Theta_0)_{\ee}|$, the minimum signal over
the edges present in the graph. (For edge $\ee=\{j,k\}$, let
$(\Theta_0)_{\ee}=(\Theta_0)_{jk}=(\Theta_0)_{kj}$.)  Finally, we write
$\lambda_{\max}$ for the maximum eigenvalue of $\Theta_0$.  Observe that
the product $\sigma^2_{\max}\lambda_{\max}$ is no larger than the condition
number of $\Theta_0$ because
$1/\lambda_{\min}(\Theta_0)=\lambda_{\max}(\Theta_0^{-1})\geq
\sigma^2_{\max}$.

\subsection{Main result}

Suppose that $n$ tends to infinity with the following asymptotic
assumptions on data and model:
\begin{equation}
\left\{\begin{array}{l}
\so\text{ \ is \ decomposable, \ with \ }|\so|\leq q,\\
\s^2_{\max}\lambda_{\max}\leq C,\\
p=\mathbf{O}(n^{\kappa}), \ p\rightarrow\infty,\\
\gamma_0=\gamma-(1-\frac{1}{4\kappa})>0,\\
(p+2q)\log p\times\frac{\lambda^2_{\max}}{\theta_0^2}=\mathbf{o}(n)\\\end{array}\right.
\label{eqn_Asymp}
\end{equation}
Here $C,\kappa>0$ and $\gamma$ are fixed reals, while the integers
$p,q$, the edge set $\so$, the matrix $\Theta_0$, and thus the
quantities $\s^2_{\max}$, $\lambda_{\max}$ and $\theta_0$ are
implicitly allowed to vary with $n$.  We suppress this latter
dependence on $n$ in the notation.  The `big oh' $\mathbf{O}(\cdot)$
and the `small oh' $\mathbf{o}(\cdot)$ are the Landau symbols.

\begin{mtheorem}
  Suppose that conditions (\ref{eqn_Asymp})
  hold. Let $\mathcal{E}$ be the set of all decomposable models $\ss$
  with $|\ss|\leq q$. Then with probability tending to $1$ as
  $n\rightarrow\infty$,
  $$
  \so=\arg\min_{\ss\in\mathcal{E}}\mathrm{BIC}_{\gamma}(\ss).
  $$
  That is, the extended BIC with parameter $\gamma$ selects the
  smallest true model $\so$ when applied to any subset of
  $\mathcal{E}$ containing $\so$.
\end{mtheorem}

In order to prove this theorem we use two techniques for comparing
likelihoods of different models.  Firstly, in Chen and Chen's work on
the GLM case \cite{chen:2010}, the Taylor approximation to the
log-likelihood function is used and we will proceed similarly when
comparing the smallest true model $\so$ to models $\ss$ which do not
contain $\so$.  The technique produces a lower bound on the decrease
in likelihood when the true model is replaced by a false model.  

\begin{theorem}
  \label{thm:1}
  Suppose that conditions (\ref{eqn_Asymp}) hold. Let $\mathcal{E}_1$
  be the set of models $\ss$ with $\ss\not\supset\so$ and $|\ss|\leq
  q$. Then with probability tending to $1$ as $n\rightarrow\infty$,
  $$
  l_n(\Theta_0)-l_n(\hat{\Theta}(\ss))>2q(\log p)(1+\gamma_0) \quad 
  \forall \ \ss\in\mathcal{E}_1.
  $$
\end{theorem}

Secondly, Porteous \cite{porteous:1989} shows that in the case of two
nested models which are both decomposable, the likelihood ratio (at
the maximum likelihood estimates) follows a distribution that can be
expressed exactly as a log product of Beta distributions. We will use
this to address the comparison between the model $\so$ and
decomposable models $\ss$ containing $\so$ and obtain an upper bound
on the improvement in likelihood when the true model is expanded to a
larger decomposable model.

\begin{theorem}
  \label{thm:2}
  Suppose that conditions (\ref{eqn_Asymp}) hold.  Let $\mathcal{E}_0$
  be the set of decomposable models $\ss$ with $\ss\supset\so$ and
  $|\ss|\leq q$. Then with probability tending to $1$ as
  $n\rightarrow\infty$,
  $$
  l_n(\hat{\Theta}(\ss))-l_n(\hat{\Theta}(\so))<
  2(1+\gamma_0)(|\ss|-|\so|)\log p \quad  \forall 
  \ss\in\mathcal{E}_0\backslash\{\so\}.
  $$
\end{theorem}

\begin{proof}[Proof of the Main Theorem]
  With probability tending to $1$ as $n\rightarrow\infty$, both of the
  conclusions of Theorems~\ref{thm:1} and \ref{thm:2} hold.
  We will show that both conclusions holding simultaneously implies
  the desired result. 
  
  Observe that $\mathcal{E}\subset \mathcal{E}_0\cup\mathcal{E}_1$.
  Choose any $\ss\in\mathcal{E}\backslash\{\so\}$. If
  $\ss\in\mathcal{E}_0$, then (by Theorem 2):
  $$
  \mathrm{BIC}_{\gamma}(\ss)-\mathrm{BIC}_{\gamma}(\so)=
  -2(l_n(\hat{\Theta}(\ss))-l_n(\hat{\Theta}(\so)))+4(1+\gamma_0)(|\ss|-|\so|)\log p>0.
  $$
  If instead $\ss\in\mathcal{E}_1$, then (by Theorem 1, since $|\so|\leq q$):
  \begin{align*}
  \mathrm{BIC}_{\gamma}(\ss)-\mathrm{BIC}_{\gamma}(\so)&=
  -2(l_n(\hat{\Theta}(\ss))-l_n(\hat{\Theta}(\so)))+
  4(1+\gamma_0)(|\ss|-|\so|)\log p>0.
  \end{align*}

Therefore, for any $\ss\in\mathcal{E}\backslash\{\so\}$,
$\mathrm{BIC}_{\gamma}(\ss)>\mathrm{BIC}_{\gamma}(\so)$, which yields
the desired result.
\end{proof}

Some details on the proofs of Theorems~\ref{thm:1} and \ref{thm:2} are
given in Section~\ref{sec:appendix}.

\section{Simulations}
\label{section_Sim}

In this section, we demonstrate that the EBIC with positive $\gamma$
indeed leads to better model selection properties in practically
relevant settings.  We let $n$ grow, set $p\propto n^{\kappa}$ for
various values of $\kappa$, and apply the EBIC with
$\gamma\in\{0,0.5,1\}$ similarly to the choice made in the regression
context by \cite{chen:2008}.  As mentioned in the introduction, we
first use the graphical lasso of \cite{friedman:2008} (as implemented
in the `glasso' package for R) to define a small set of models to
consider (details given below).  From the selected set we choose the
model with the lowest EBIC.  This is repeated for $100$ trials for
each combination of values of $n,p,\gamma$ in each scaling scenario.
For each case, the average positive selection rate (PSR) and false
discovery rate (FDR) are computed.

We recall that the graphical lasso places an $\ell_1$ penalty on the
inverse covariance matrix.  Given a penalty $\rho\ge 0$, we obtain
the estimate
\begin{equation}
  \label{eq:ell1}
  \hat{\Theta}_{\rho}=\arg\min_{\Theta}
  -l_n(\Theta)+\rho\|\Theta\|_1.
\end{equation}
(Here we may define $\|\Theta\|_1$ as the sum of absolute values of
all entries, or only of off-diagonal entries; both variants are
common).  The $\ell_1$ penalty promotes zeros in the estimated inverse
covariance matrix $\hat{\Theta}_{\rho}$; increasing the penalty
yields an increase in sparsity.  The `glasso path', that is, the set
of models recovered over the full range of penalties
$\rho\in[0,\infty)$, gives a small set of models which, roughly,
include the `best' models at various levels of sparsity.  We may
therefore apply the EBIC to this manageably small set of models
(without further restriction to decomposable models).  Consistency
results on the graphical lasso require the penalty $\rho$ to
satisfy bounds that involve measures of regularity in the unknown
matrix $\Theta_0$; see \cite{ravikumar:2008}.  Minimizing the EBIC can
be viewed as a data-driven method of tuning $\rho$, one that does
not require creation of test data.

While
cross-validation does not generally have consistency properties for
model selection (see \cite{shao:1993}), it is nevertheless interesting to compare our method to cross-validation.  For the considered simulated data, we start with the set
of models from the `glasso path', as before, and then perform 100-fold 
cross-validation. For each model and each choice of training set and test
set, we fit the model to the training set and then evaluate its performance
on each sample in the test set, by measuring error in predicting each 
individual node conditional on the other nodes and then taking the sum of
the squared errors. We note that this
method is computationally much more intensive than the BIC or EBIC,
because models need to be fitted many more times.

\begin{figure}[t]
  \begin{center}
    \includegraphics[width=0.3\linewidth]{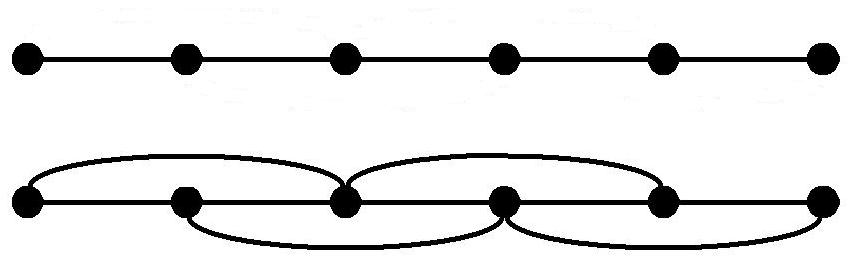}
  \end{center}
  \caption{The chain (top) and the `double chain' (bottom) on 6 nodes.} 
  \label{fig_DoubleChain}
\end{figure}

\subsection{Design}

In our simulations, we examine the EBIC as applied to the case
where the graph is a chain with node $j$ being connected to nodes
$j-1,j+1$, and to the `double chain', where node $j$ is connected to
nodes $j-2,j-1,j+1,j+2$.  Figure~\ref{fig_DoubleChain} shows examples
of the two types of graphs, which have on the order of $p$ and $2p$
edges, respectively.  For both the chain and the double chain, we
investigate four different scaling scenarios, with the exponent
$\kappa$ selected from $\{0.5,0.9,1,1.1\}$. In each 
scenario, we test $n=100,200,400,800$, and define $p\propto
n^{\kappa}$ with the constant of proportionality chosen such that
$p=10$ when $n=100$ for better comparability.

In the case of a chain, the true inverse covariance matrix $\Theta_0$
is tridiagonal with all diagonal entries $(\Theta_0)_{j,j}$ set
equal to 1, and the entries $(\Theta_0)_{j,j+1}=(\Theta_0)_{j+1,j}$
that are next to the main diagonal equal to 0.3.  For the double
chain, $\Theta_0$ has all diagonal entries equal to 1, the entries
next to the main diagonal are
$(\Theta_0)_{j,j+1}=(\Theta_0)_{j+1,j}=0.2$ and the remaining non-zero
entries are $(\Theta_0)_{j,j+2}=(\Theta_0)_{j+2,j}=0.1$.  In both
cases, the choices result in values for $\theta_0$, $\s^2_{\max}$ and
$\lambda_{\max}$ that are bounded uniformly in the matrix size $p$.

For each data set generated from $N(0,\Theta_0^{-1})$, we use the
`glasso' package \cite{friedman:2008} in R to compute the `glasso path'.  We choose 100
penalty values $\rho$ which are logarithmically evenly spaced
between $\rho_{\max}$ (the smallest value which will result in a no-edge
model) and $\rho_{\max}/100$. At each penalty
value $\rho$, we compute $\hat{\Theta}_{\rho}$ from
(\ref{eq:ell1}) and define the model $\ss_\rho$ based on this
estimate's support.  The R routine also allows us to compute the
unpenalized maximum likelihood estimate $\hat\Theta(\ss_\rho)$.  We
may then readily compute the EBIC from (\ref{eq:graphical-ebic}).
There is no guarantee that this procedure will find the model with 
the lowest EBIC along the full `glasso path', let alone among the 
space of all possible models of size $\leq q$.  Nonetheless, it
 serves as a fast way to
select a model without any manual tuning.

\subsection{Results}

{\it Chain graph:} The results for the chain graph are displayed in
Figure~\ref{fig_ChainScalingResults}. The figure shows the positive
selection rate (PSR) and false discovery rate (FDR) in the four
scaling scenarios. We observe that, for the larger sample sizes, the
recovery of the non-zero coefficients is perfect or nearly perfect for
all three values of $\gamma$; however, the FDR rate is noticeably
better for the positive values of $\gamma$, especially for higher
scaling exponents $\kappa$. Therefore, for moderately large $n$, the
EBIC with $\gamma=0.5$ or $\gamma=1$ performs very well, while the
ordinary $\mathrm{BIC}_{0}$ produces a non-trivial amount of false
positives. For 100-fold cross-validation, while the PSR is initially
slightly higher, the growing FDR demonstrates the extreme inconsistency
of this method in the given setting.

{\it Double chain graph:} The results for the double chain graph are
displayed in Figure~\ref{fig_DoubleChainScalingResults}. In each of
the four scaling scenarios for this case, we see a noticeable decline
in the PSR as $\gamma$ increases.  Nonetheless, for each value of
$\gamma$, the PSR increases as $n$ and $p$ grow.  Furthermore, the FDR
for the ordinary $\mathrm{BIC}_{0}$ is again noticeably higher than
for the positive values of $\gamma$, and in the scaling scenarios
$\kappa\ge 0.9$, the FDR for $\mathrm{BIC}_{0}$ is actually increasing
as $n$ and $p$ grow, suggesting that asymptotic consistency may not
hold in these cases, as is supported by our theoretical results. 100-fold
cross-validation shows significantly better PSR than the BIC and EBIC methods,
but the FDR is again extremely high and increases quickly as the model grows,
which shows the unreliability of cross-validation in this setting. Similarly to what Chen and Chen \cite{chen:2008} conclude for the regression case,
it appears that the EBIC with parameter $\gamma=0.5$ performs well.
Although the PSR is necessarily lower than with $\gamma=0$, the FDR is
quite low and decreasing as $n$ and $p$ grow, as desired.

For both types of simulations, the results demonstrate the trade-off inherent in choosing $\gamma$
in the finite (non-asymptotic) setting. For low values of $\gamma$, we are more likely to
obtain a good (high) positive selection rate. For higher values
of $\gamma$, we are more likely to obtain a good (low) false discovery rate. (In the proofs given in Section~\ref{sec:appendix}, this corresponds to assumptions~(\ref{eqn_Assumption1}) and~(\ref{eqn_Assumption2})). However, asymptotically, the conditions~(\ref{eqn_Asymp}) guarantee consistency, meaning that the trade-off becomes irrelevant for large $n$ and $p$. In the finite case, $\gamma=0.5$ seems to be a good compromise in simulations, but the question of determining the best value of $\gamma$ in general settings is an open question. Nonetheless, this method offers guaranteed asymptotic consistency
for (known) values of $\gamma$ depending only on $n$ and $p$.

\begin{figure}[t]
\begin{center}
              \includegraphics[width=0.48\linewidth]{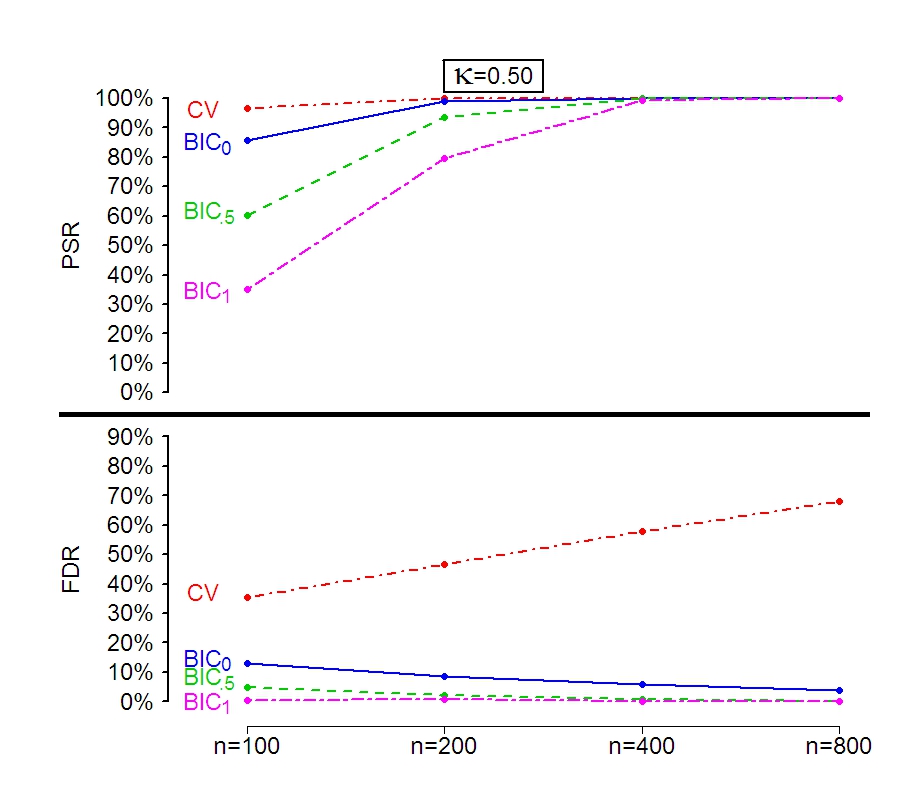}
              \includegraphics[width=0.48\linewidth]{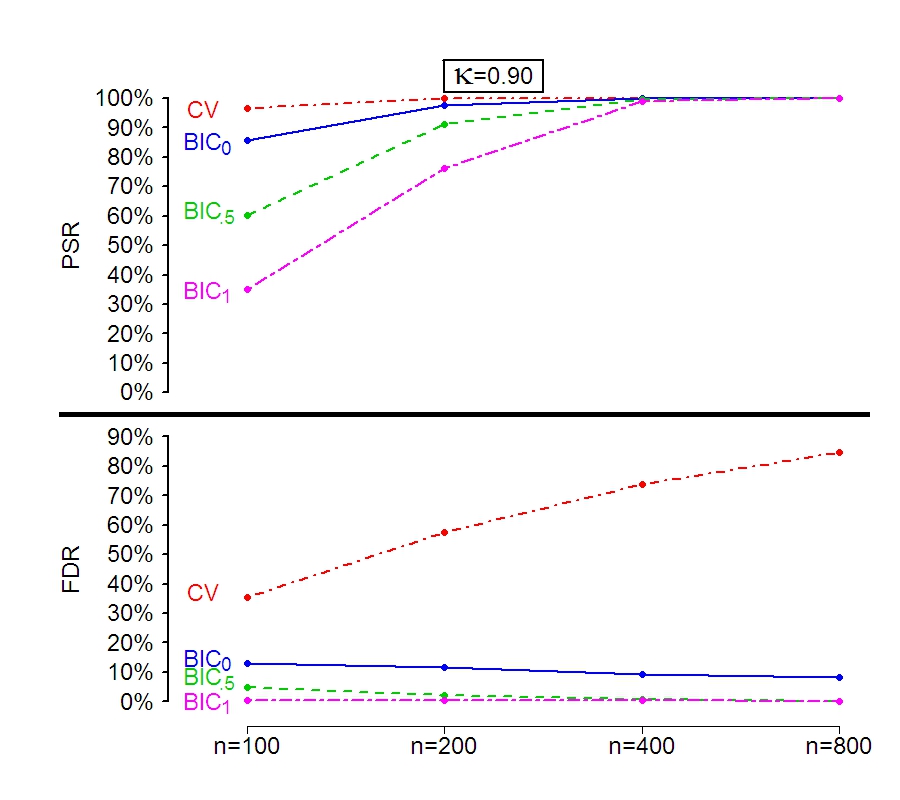}
              \includegraphics[width=0.48\linewidth]{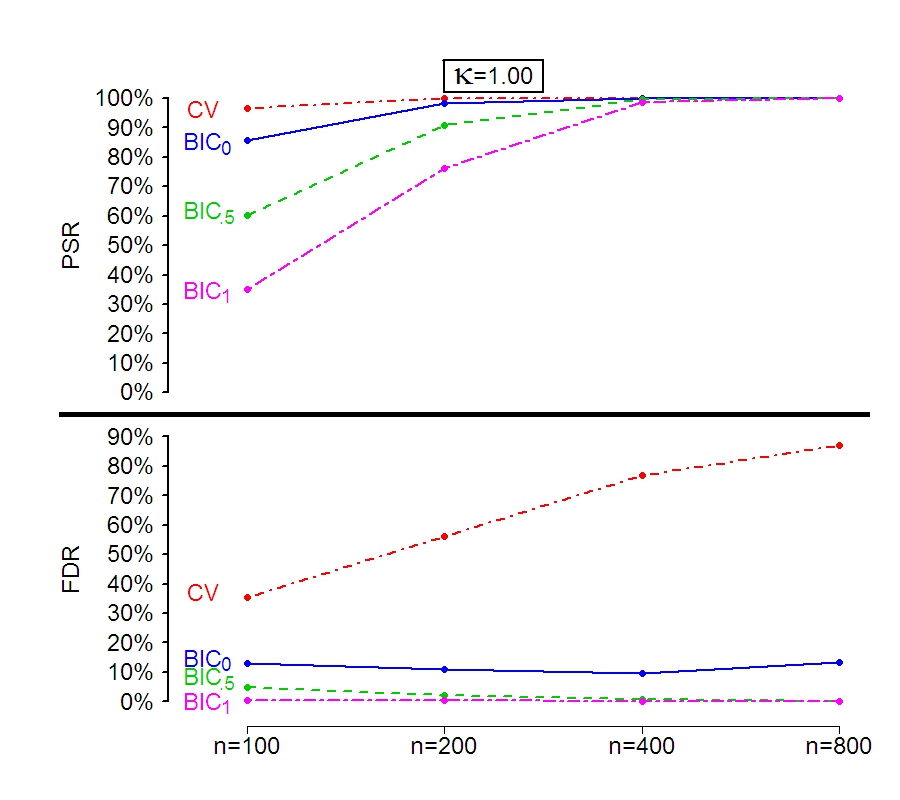}
              \includegraphics[width=0.48\linewidth]{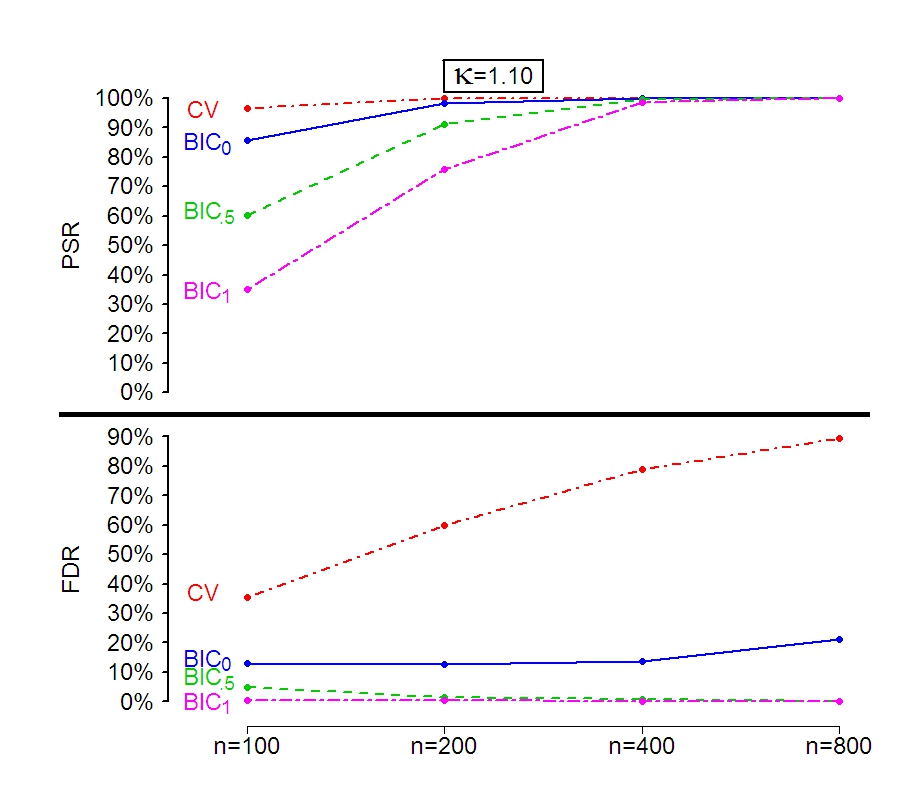}
\end{center}
\caption{Simulation results when the true graph is a chain.}
\label{fig_ChainScalingResults}
\end{figure}

\begin{figure}[t]
\begin{center}
              \includegraphics[width=0.48\linewidth]{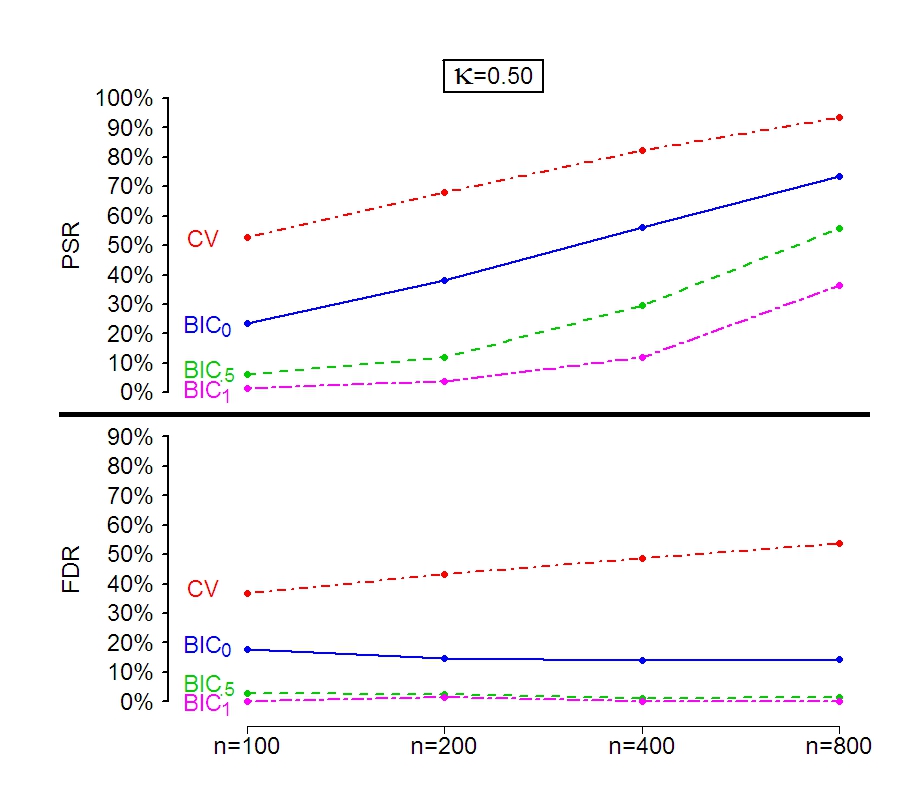}
              \includegraphics[width=0.48\linewidth]{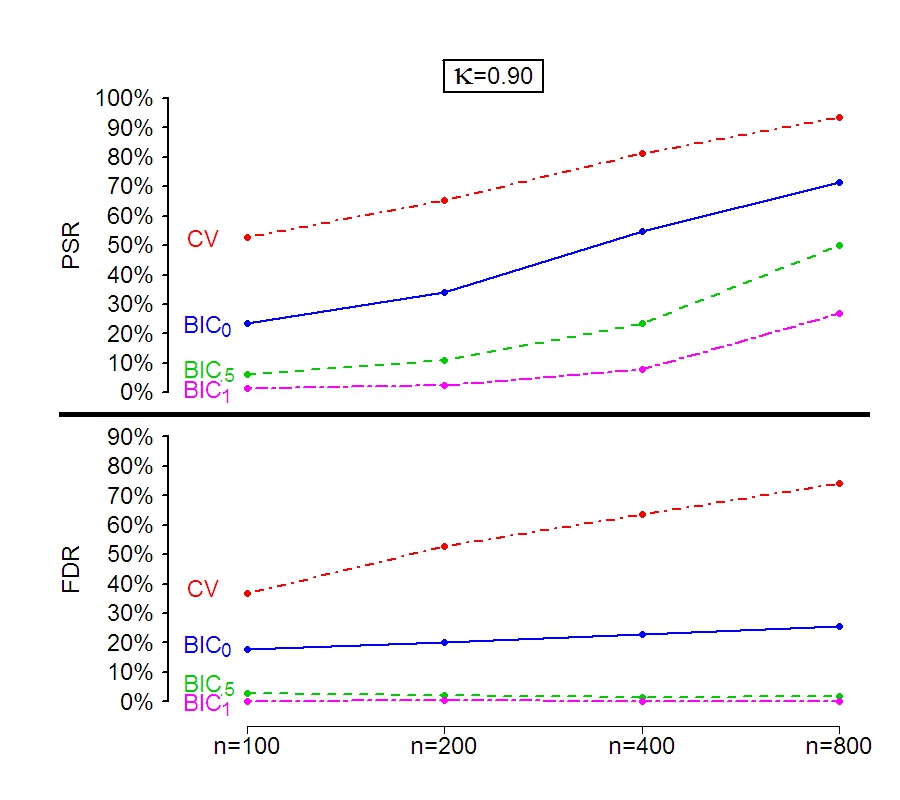}
              \includegraphics[width=0.48\linewidth]{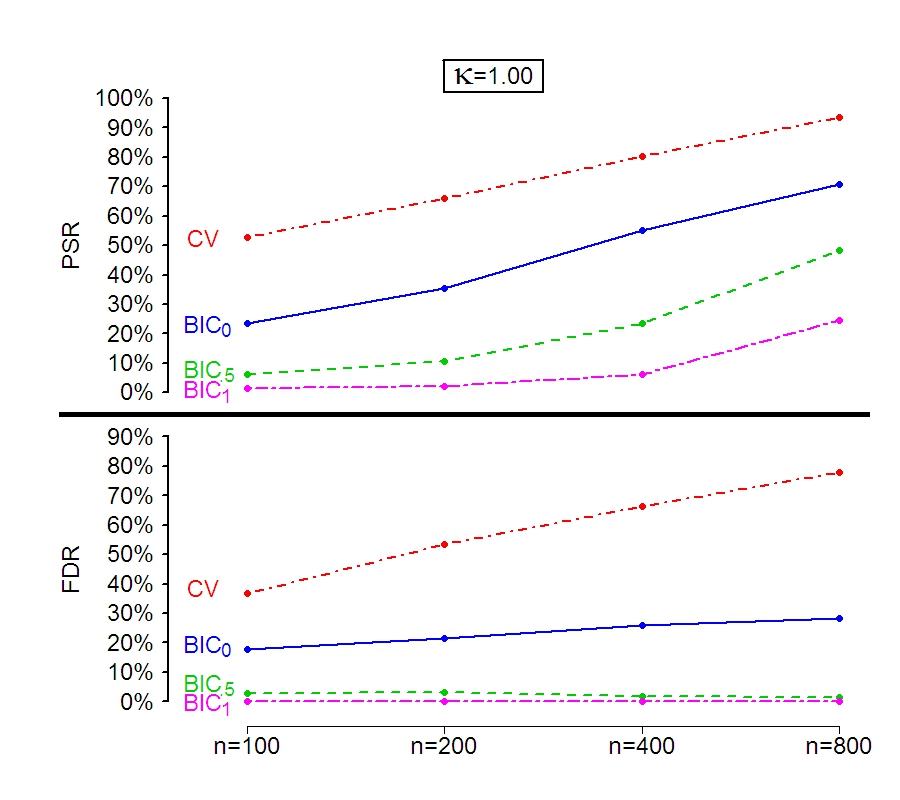}
              \includegraphics[width=0.48\linewidth]{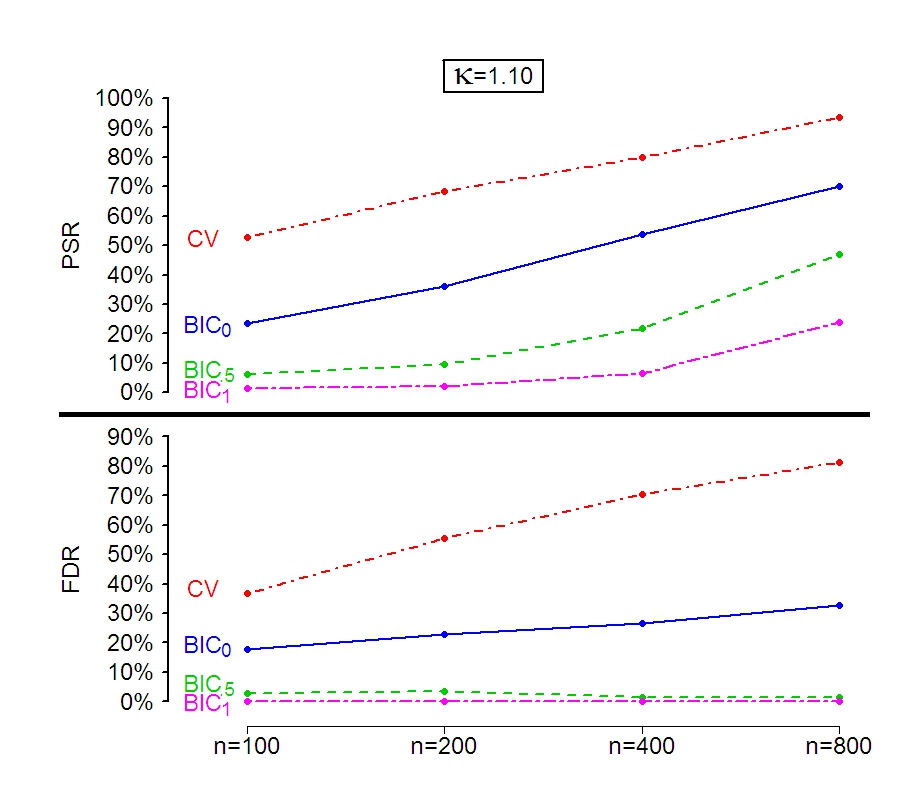}
\end{center}
\caption{Simulation results when the true graph is a `double chain'.}
\label{fig_DoubleChainScalingResults}
\end{figure}

\section{Discussion}
\label{section_Disc}

We have proposed the use of an extended Bayesian information criterion
for multivariate data generated by sparse graphical models.  Our main
result gives a specific scaling for the number of variables $p$, the
sample size $n$, the bound on the number of edges $q$, and other
technical quantities relating to the true model, which will ensure
asymptotic consistency.  Our simulation study demonstrates the the
practical potential of the extended BIC, particularly as a way to tune
the graphical lasso. The results show that the extended BIC with
positive $\gamma$ gives strong improvement in false discovery rate
over the classical BIC, and even more so over cross-validation,
while showing comparable positive selection
rate for the chain, where all the signals are fairly strong, and
noticeably lower, but steadily increasing, positive selection rate for
the double chain with a large number of weaker signals.

\section{Proofs}
\label{sec:appendix}

We now sketch proofs of non-asymptotic versions of
Theorems~\ref{thm:1} and \ref{thm:2}, which are formulated as
Theorems~\ref{thm:1-nonasy} and \ref{thm:2-nonasy}. (Full technical details are given in the Appendix.) We also give a
non-asymptotic formulation of the Main Theorem; see
Theorem~\ref{mthm:nonasy}.  
In the non-asymptotic approach, we treat all
quantities as fixed (e.g.\ $n,p,q$, etc.) and state precise
assumptions on those quantities, and then give an explicit lower bound
on the probability of the extended BIC recovering the model $\so$
exactly. We do this to give an intuition for the magnitude of the
sample size $n$ necessary for a good chance of exact recovery in a
given setting but due to the proof techniques, the resulting
implications about sample size are extremely conservative.

\subsection{Preliminaries}

We begin by stating two lemmas that are used in the proof of the main
result, but are also more generally interesting as tools for precise
bounds on Gaussian and chi-square distributions. First, Cai \cite[Lemma 4]{cai:2002} proves the following chi-square
bound. For any $n\geq 1,\lambda>0$,
$$
P\{\chi^2_n>n(1+\lambda)\}\leq \frac{1}{\lambda\sqrt{\pi
    n}}e^{-\frac{n}{2}(\lambda-\log(1+\lambda))}.
$$
We can give an analagous left-tail upper bound. The proof is 
similar to Cai's proof and omitted here. We will refer to these two
bounds together as (CSB).

\begin{lemma}
  \label{lem:1}
  For any $\lambda>0$, for $n$ such that $n\geq 4\lambda^{-2}+1$,
  \[
  P\{\chi^2_n<n(1-\lambda)\}\leq
  \frac{1}{\lambda\sqrt{\pi(n-1)}}e^{\frac{n-1}{2}(\lambda+\log(1-\lambda))}.
  \]
\end{lemma}

Second, we give a distributional result about the sample correlation
when sampling from a bivariate normal distribution.

\begin{lemma}
  \label{lem:2} Suppose $(X_1,Y_1),\dots,(X_n,Y_n)$ are independent
  draws from a bivariate normal distribution with zero mean, variances
  equal to one and covariance $\rho$.  Then the following
  distributional equivalence holds, where $A$ and $B$ are independent
  $\chi^2_n$ variables:
  $$
  \sum_{i=1}^n
  (X_iY_i-\rho)\;\stackrel{\mathcal{D}}{=}\;\frac{1+\rho}{2}(A-n)-\frac{1-\rho}{2}(B-n).
  $$
\end{lemma}

\begin{proof}
  Let $A_1,B_1,A_2,B_2,\dots,A_n,B_n$ be independent standard normal
  random variables. Define:
  $$
  X_i=\sqrt{\frac{1+\rho}{2}}A_i+\sqrt{\frac{1-\rho}{2}}B_i; \ 
  Y_i=\sqrt{\frac{1+\rho}{2}}A_i-\sqrt{\frac{1-\rho}{2}}B_i; \ 
  A=\sum_{i=1}^n A_i^2; \ B=\sum_{i=1}^n B_i^2.
  $$
  Then the variables
  $X_1,Y_1,X_2,Y_2,\dots,X_n,Y_n$ have the desired joint distribution,
  and $A,B$ are independent $\chi^2_n$ variables. The claim follows from writing $\sum_i X_iY_i$ in terms of $A$ and $B$.\end{proof}

\subsection{Non-asymptotic versions of the theorems}

We assume the following two conditions, where $\e_0,\e_1>0$, $C\geq
\s^2_{\max}\lambda_{\max}$, $\kappa=\log_n p$, and
$\gamma_0=\gamma-(1-\frac{1}{4\kappa})$:
\begin{align}
  &
  \frac{(p+2q)\log p}{n}\times\frac{\lambda^2_{\max}}{\theta_0^2}\leq
  \frac{1}{3200\max\{1+\gamma_0,\left(1+\frac{\epsilon_1}{2}\right)C^2\}}
  \label{eqn_Assumption1}\\
  &
  2(\sqrt{1+\gamma_0}-1)-\frac{\log \log
    p+\log(4\sqrt{1+\gamma_0})+1}{2\log p}\geq \e_0
  \label{eqn_Assumption2}
\end{align}

\begin{theorem}
  \label{thm:1-nonasy} Suppose assumption (\ref{eqn_Assumption1}) holds.
  Then with probability at least $1-\frac{1}{\sqrt{\pi\log
      p}}p^{-\e_1}$, for all $\ss\not\supset\so$ with $|\ss|\leq q$,
  $$
  l_n(\Theta_0)-l_n(\hat{\Theta}(\ss))>2q(\log p)(1+\gamma_0).
  $$
\end{theorem}
\begin{proof}
  We sketch a proof along the lines of the proof of Theorem 2
  in \cite{chen:2010}, using Taylor series centered at the true
  $\Theta_0$ to approximate the likelihood at $\hat{\Theta}(\ss)$.
  The score and the negative Hessian of the log-likelihood
  function in (\ref{eq:loglik}) are
  \begin{align*}
    s_n(\Theta) &=
    \frac{d}{d\Theta}l_n(\Theta)=\frac{n}{2}\left(\Theta^{-1}-S\right),&
    H_n(\Theta)
    &=-\frac{d}{d\Theta}s_n(\Theta)=\frac{n}{2}\Theta^{-1}\otimes\Theta^{-1}.
  \end{align*}  
  Here, the symbol $\otimes$ denotes the Kronecker product of
  matrices. Note that, while we require $\Theta$ to be symmetric
  positive definite, this is not reflected in the derivatives above.
  We adopt this convention for the notational convenience in the
  sequel.
  
  Next, observe that $\hat{\Theta}(\ss)$ has support on
  $\Delta\cup\so\cup\ss$, and that by definition of $\theta_0$, we
  have the lower bound $|\hat{\Theta}(\ss)-\Theta_0|_F\geq\theta_0$
  in terms of the Frobenius norm.  By concavity of the log-likelihood
  function, it suffices to show that the desired inequality holds for
  all $\Theta$ with support on $\Delta\cup\so\cup\ss$ with
  $|\Theta-\Theta_0|_F=\theta_0$.  By Taylor expansion, for some
  $\tilde{\Theta}$ on the path from $\Theta_0$ to $\Theta$, we have:
  $$
  l_n(\Theta)-l_n(\Theta_0)=\mathrm{vec}(\Theta-\Theta_0)^Ts_n(\Theta_0)-\frac{1}{2}\mathrm{vec}(\Theta-\Theta_0)^TH_n(\tilde{\Theta})\mathrm{vec}(\Theta-\Theta_0).
  $$
  Next, by (CSB) and Lemma~\ref{lem:2}, with probability at least
  $1-\frac{1}{\sqrt{\pi\log p}}e^{-\e_1\log p}$, the following bound
  holds for all edges $\ee$ in the complete graph (we omit the
  details):  
  $$
  (s_n(\Theta_0))_{\ee}^2\leq 6\s^4_{\max}(2+\e_1)n\log p.
  $$
  Now assume that this bound holds for all edges. Fix some $\ss$ as
  above, and fix $\Theta$ with support on $\Delta\cup\so\cup\ss$, with
  $|\Theta-\Theta_0|=\theta_0$. Note that the support has at most
  $(p+2q)$ entries.  Therefore,
  $$
  |\mathrm{vec}(\Theta-\Theta_0)^Ts_n(\Theta_0)|^2\leq\theta_0^2(p+2q)\times
  6\s^4_{\max}(2+\e_1)n\log p.
  $$
  Furthermore, the eigenvalues of $\Theta$ are bounded by
  $\lambda_{\max}+\theta_0\leq 2\lambda_{\max}$, and so by properties
  of Kronecker products, the minimum eigenvalue of
  $H_n(\tilde{\Theta})$ is at least
  $\frac{n}{2}(2\lambda_{\max})^{-2}$.  We conclude that  
  $$
  l_n(\Theta)-l_n(\Theta_0)\leq\sqrt{\theta_0^2(p+2q)\times
    6\s^4_{\max}(2+\e_1)n\log p}-\frac{1}{2}\theta_0^2\times
  \frac{n}{2}(2\lambda_{\max})^{-2}.
  $$
  Combining this bound with our assumptions above, we obtain the
  desired result.
\end{proof}

\begin{theorem}
  \label{thm:2-nonasy} 
  Suppose additionally that assumption (\ref{eqn_Assumption2}) holds
  (in particular, this implies that $\gamma>1-\frac{1}{4\kappa}$).
  Then with probability at least $1-\frac{1}{4\sqrt{\pi}\log
    p}\frac{p^{-\e_0}}{1-p^{-\e_0}}$, for all decomposable models
  $\ss$ such that $\ss\supsetneq \so$ and $|\ss|\leq q$,
  $$
  l_n(\hat{\Theta}(\ss))-l_n(\hat{\Theta}(\so))<
  2(1+\gamma_0)(|\ss|-|\so|)\log p.
  $$
\end{theorem}
\begin{proof}
  First, fix a single such model $\ss$, and define $m=|\ss|-|\so|$. By
  \cite{porteous:1989,eriksen:1996}, $  l_n(\hat{\Theta}(\ss))-l_n(\hat{\Theta}(\so))$ is distributed as $
  -\frac{n}{2}\log\left(\prod_{i=1}^m
    B_i\right)$,
  where $B_i\sim \mathit{Beta}(\frac{n-c_i}{2},\frac{1}{2})$ are independent
  random variables and the constants $c_1,\dots,c_m$ are bounded by
  $1$ less than the maximal clique size of the graph given by model
  $\ss$, implying $c_i\leq\sqrt{2q}$ for each $i$.  Also shown in \cite{porteous:1989} is the stochastic
  inequality $-\log(B_i)\leq \frac{1}{n-c_i-1}\chi^2_1$. It
  follows that,  stochastically,
  $$
  l_n(\hat{\Theta}(\ss))-l_n(\hat{\Theta}(\so))\leq
  \frac{n}{2}\times\frac{1}{n-\sqrt{2q}-1}\chi^2_m.
  $$
  Finally, combining the assumptions on $n,p,q$ and the (CSB)
  inequalities, we obtain:
  $$
  P\{l_n(\hat{\Theta}(\ss))-l_n(\hat{\Theta}(\so))\geq
  2(1+\gamma_0)m\log(p)\}\leq \frac{1}{4\sqrt{\pi}\log
    p}e^{-\frac{m}{2}(4(1+\frac{\e_0}{2})\log p)}.
  $$
  
  Next, note that the number of models $|\ss|$ with $\ss\supset\so$
  and $|\ss|-|\so|=m$ is bounded by $p^{2m}$. Taking the union bound
  over all choices of $m$ and all choices of $\ss$ with that given
  $m$, we obtain that the desired result holds with the desired probability.
\end{proof}

We are now ready to give a non-asymptotic version of the Main Theorem.
For its proof apply the union bound to the statements in
Theorems~\ref{thm:1-nonasy} and \ref{thm:2-nonasy}, as in the
asymptotic proof given in section~\ref{section_Main}.

\begin{theorem}
  \label{mthm:nonasy}
  Suppose assumptions (\ref{eqn_Assumption1}) and
  (\ref{eqn_Assumption2}) hold. Let $\mathcal{E}$ be the set of
  subsets $\ss$ of edges between the $p$ nodes, satisfying $|\ss|\leq
  q$ and representing a decomposable model. Then it holds with
  probability at least $1-\frac{1}{4\sqrt{\pi}\log
    p}\frac{p^{-\e_0}}{1-p^{-\e_0}}-\frac{1}{\sqrt{\pi\log
      p}}p^{-\e_1}$ that
  $$
  \so=\arg\min_{\ss\in\mathcal{E}}\mathrm{BIC}_{\gamma}(\ss).
  $$
  That is, the extended BIC with parameter $\gamma$ selects the
  smallest true model.
\end{theorem}

Finally, we note that translating the above to the asymptotic version
of the result is simple.  If the conditions (\ref{eqn_Asymp}) hold,
then for sufficiently large $n$ (and thus sufficiently large
$p$), assumptions (\ref{eqn_Assumption1}) and (\ref{eqn_Assumption2})
hold. Furthermore, although we may not have the exact equality
$\kappa=\log_n p$, we will have $\log_n p\rightarrow \kappa$; this
limit will be sufficient for the necessary inequalities to hold for
sufficiently large $n$. The proofs then follow from the non-asymptotic
results.

\bibliographystyle{plainnat} 
\bibliography{graphicalEBIC_arxiv}

\begin{thebibliography}{10}

\bibitem{lauritzen:1996}
Steffen~L. Lauritzen.
\newblock {\em Graphical models}, volume~17 of {\em Oxford Statistical Science
  Series}.
\newblock The Clarendon Press Oxford University Press, New York, 1996.
\newblock Oxford Science Publications.

\bibitem{friedman:2008}
Jerome Friedman, Trevor Hastie, and Robert Tibshirani.
\newblock Sparse inverse covariance estimation with the graphical lasso.
\newblock {\em Biostatistics}, 9(3):432--441, 2008.

\bibitem{chen:2008}
Jiahua Chen and Zehua Chen.
\newblock Extended {B}ayesian information criterion for model selection with
  large model space.
\newblock {\em Biometrika}, 95:759--771, 2008.

\bibitem{schwarz:1978}
Gideon Schwarz.
\newblock Estimating the dimension of a model.
\newblock {\em Ann. Statist.}, 6(2):461--464, 1978.

\bibitem{bogdan:2004}
Malgorzata Bogdan, Jayanta~K. Ghosh, and R.~W. Doerge.
\newblock Modifying the {S}chwarz {B}ayesian information criterion to locate
  multiple interacting quantitative trait loci.
\newblock {\em Genetics}, 167:989--999, 2004.

\bibitem{chen:2010}
Jiahua Chen and Zehua Chen.
\newblock Extended {BIC} for small-$n$-large-$p$ sparse {GLM}.
\newblock Preprint.

\bibitem{ravikumar:2008}
Pradeep Ravikumar, Martin~J. Wainwright, Garvesh Raskutti, and Bin Yu.
\newblock High-dimensional covariance estimation by minimizing
  $\ell_1$-penalized log-determinant divergence.
\newblock arXiv:0811.3628, 2008.

\bibitem{porteous:1989}
B.~T. Porteous.
\newblock Stochastic inequalities relating a class of log-likelihood ratio
  statistics to their asymptotic {$\chi^2$} distribution.
\newblock {\em Ann. Statist.}, 17(4):1723--1734, 1989.

\bibitem{shao:1993}
Jun Shao.
\newblock Linear model selection by cross-validation.
\newblock {\em J. Amer. Statist. Assoc.}, 88(422):486--494, 1993.

\bibitem{cai:2002}
T.~Tony Cai.
\newblock On block thresholding in wavelet regression: adaptivity, block size,
  and threshold level.
\newblock {\em Statist. Sinica}, 12(4):1241--1273, 2002.

\bibitem{eriksen:1996}
P.~Svante Eriksen.
\newblock Tests in covariance selection models.
\newblock {\em Scand. J. Statist.}, 23(3):275--284, 1996.

\end{thebibliography}

\clearpage
\begin{appendix}

\section{Appendix}
This section gives the proof for Lemma~\ref{lem:1}, and fills in the details of the proofs of Theorems~\ref{thm:1-nonasy} and~\ref{thm:2-nonasy}.

{\bf Lemma~\ref{lem:1}}{\it
For any $\lambda\in(0,1)$, for any $n\geq 4\lambda^{-2}+1$,
$$P\{\chi^2_n<n(1-\lambda)\}\leq \frac{1}{\lambda\sqrt{\pi (n-1)}}e^{-\frac{n-1}{2}(\lambda-\log(1+\lambda))}\;\;.$$
}

\begin{remark}
We note that some lower bound on $n$ is intuitively necessary in order to be able to bound the `left tail', because the mode of the $\chi^2_n$ distribution is at $x=n-2$ (for $n\geq 2$). If $\lambda$ is very close to zero, then the `left tail' ($\chi^2_n\in[0,n(1-\lambda)]$) actually includes the mode $x=n-2\leq n(1-\lambda)$; therefore, we could not hope to get an exponentially small probability for being in the tail. However, this intuitive explanation suggests that we should have $n\geq \mathbf{O}(\lambda^{-1})$; perhaps the bound in this lemma could be tightened.\end{remark}

We first prove a preliminary lemma:

\begin{alemma}\label{prelim_lemma}For any $\lambda>0$, for any $n\geq 4\lambda^{-2}+1$,
$$P\{\chi^2_{n+1}<(n+1)(1-\lambda)\}\leq P\{\chi^2_n\leq n(1-\lambda)\}\;\;.$$
\end{alemma}
\begin{proof} Let $f_n$ denote the density function for $\chi^2_n$, and let $\tilde{f}_n$ denote the density function for $\frac{1}{n}\chi^2_n$. Then, using $y=x/n$, we get:
$$f_n(x)=\frac{1}{2^{n/2}\Gamma(n/2)}x^{n/2-1}e^{-x/2} \ \Rightarrow \ \tilde{f}_n(y)=\frac{1}{2^{n/2}\Gamma(n/2)}y^{n/2-1}e^{-ny/2}n^{n/2}\;\;.$$
So,
$$\tilde{f}_{n+1}(y)=\tilde{f}_n(y)\times \left[\sqrt{\frac{n+1}{2}}\frac{\Gamma(n/2)}{\Gamma((n+1)/2)}\left(\frac{n+1}{n}\right)^{n/2}\sqrt{ye^{-y}}\right]\;\;.$$
First, note that $ye^{-y}$ is an increasing function for $y<1$, and therefore
$$y\in[0,1-\lambda] \ \Rightarrow \ ye^{-y}\leq (1-\lambda)e^{-(1-\lambda)}\leq e^{-1}\left(1-\frac{\lambda^2}{2}\right)\;\;.$$
(Here the last inequality is from the Taylor series). Next, since $\log\Gamma(x)$ is a convex function (where $x>0$), and since $\Gamma((n+1)/2)=\Gamma((n-1)/2)\times\frac{n-1}{2}$, we see that
$$ \frac{\Gamma((n+1)/2)}{\Gamma(n/2)}\geq\sqrt{\frac{n-1}{2}}\;\;.$$
Finally, it is a fact that $(1+\frac{1}{n})^n\leq e$. Putting the above bounds together, and assuming that $y\in[0,1-\lambda]$, we obtain
$$\tilde{f}_{n+1}(y)\leq \tilde{f}_n(y)\times \left[\sqrt{\frac{n+1}{2}}\sqrt{\frac{2}{n-1}}\sqrt{e}\sqrt{e^{-1}\left(1-\frac{\lambda^2}{2}\right)}\right]$$
$$=\tilde{f}_n(y)\times \left[\sqrt{\frac{n+1}{n-1}}\sqrt{1-\frac{\lambda^2}{2}}\right]\;\;.$$
Since we require $n\geq 4\lambda^{-2}+1$, the quantity in the brackets is at most $1$, and so
$$\tilde{f}_{n+1}(y)\leq \tilde{f}_n(y) \ \forall \ y\in[0,1-\lambda]\;\;.$$
Therefore,
$$P\left\{\frac{1}{n+1}\chi^2_{n+1}<(1-\lambda)\right\}\leq P\left\{\frac{1}{n}\chi^2_n<(1-\lambda)\right\}\;\;.$$
\end{proof}

Now we prove Lemma~\ref{lem:1}.
\begin{proof}
First suppose that $n$ is even. Let $f_n$ denote the density function of the $\chi^2_n$ distribution. From~\cite{cai:2002}, if $n>2$,
$$P\{\chi^2_n<x\}=1-2f_n(x)-P\{\chi^2_{n-2}>x\}=-2f_n(x)+P\{\chi^2_{n-2}<x\}\;\;.$$
Iterating this identity, we get
\begin{eqnarray*}P\{\chi^2_n<x\}&=&P\{\chi^2_2<x\}-2f_n(x)-2f_{n-2}(x)-\dots-2f_4(x)\\
&=&1-e^{-\frac{x}{2}}-2\sum_{k=1}^{n/2-1}f_{2k+2}(x)\\
&=&1-e^{-\frac{x}{2}}-2\sum_{k=1}^{n/2-1}\frac{1}{2^{k+1}\Gamma(k+1)}x^ke^{-\frac{x}{2}}\\
&=&1-e^{-\frac{x}{2}}\left(\sum_{k=0}^{n/2-1}\frac{1}{2^kk!}x^k\right)\\
&=&1-e^{-\frac{x}{2}}\left(\sum_{k=0}^{\infty}\frac{(x/2)^k}{k!}-\sum_{k=n/2}^{\infty}\frac{(x/2)^k}{k!}\right)\\
&=&1-e^{-\frac{x}{2}}\left(e^{\frac{x}{2}}-\sum_{k=n/2}^{\infty}\frac{(x/2)^k}{k!}\right)\\
&=&e^{-\frac{x}{2}}\sum_{k=n/2}^{\infty}\frac{(x/2)^k}{k!}\;\;.\\
\end{eqnarray*}
Now set $x=n(1-\lambda)$ for $\lambda\in(0,1)$. We obtain
\begin{eqnarray*}P\{\chi^2_n<x\}&=&e^{-\frac{n(1-\lambda)}{2}}\sum_{k=n/2}^{\infty}\frac{(n(1-\lambda)/2)^k}{k!}\\
&\leq& e^{-\frac{n(1-\lambda)}{2}}\frac{(n/2)^{n/2}}{(n/2)!}\sum_{k=n/2}^{\infty}(1-\lambda)^k\\
&=& e^{-\frac{n(1-\lambda)}{2}}\frac{(n/2)^{n/2}}{(n/2)!}\frac{(1-\lambda)^{n/2}}{\lambda}\;\;.\\
\end{eqnarray*}
By Stirling's formula,
$$\frac{(n/2)^{n/2}}{(n/2)!}\leq \frac{e^{\frac{n}{2}}}{\sqrt{\pi n}}\;\;,$$
and so,
$$P\{\chi^2_n<n(1-\lambda)\}\leq e^{-\frac{n(1-\lambda)}{2}}\frac{e^{\frac{n}{2}}}{\sqrt{\pi n}}\frac{(1-\lambda)^{n/2}}{\lambda}=\frac{1}{\lambda \sqrt{\pi n}}e^{\frac{n}{2}(\lambda+\log (1-\lambda))}\;\;.$$

This is clearly sufficient to prove the desired bound in the case that $n$ is even. Next we turn to the odd case; let $n$ be odd. First observe that if $\lambda>1$, the statement is trivial, while if $\lambda\leq 1$, then $n\geq 4\lambda^{-2}+1\geq 5$, therefore $n-1$ is positive. By Lemma~\ref{prelim_lemma} and the expression above,
$$P\{\chi^2_n<n(1-\lambda)\}\leq P\{\chi^2_{n-1}\leq (n-1)(1-\lambda)\}\leq \frac{1}{\lambda \sqrt{\pi (n-1)}}e^{\frac{n-1}{2}(\lambda+\log (1-\lambda))}\;\;.$$

\end{proof}

Next we turn to the theorems. Recall assumptions~\ref{eqn_Assumption1} and~\ref{eqn_Assumption2}. Lemmas~\ref{scorebound} and~\ref{likelihoodbound} below are sufficient to fill in the details of Theorem~\ref{thm:1-nonasy}.

\begin{alemma}\label{scorebound} With probability at least $1-\frac{1}{\sqrt{\pi \log p}}e^{-\epsilon_1\log p}$, the following holds for all edges $\mathbf{e}$ in the complete graph:
$$(s_n(\Theta_0))_{\mathbf{e}}^2\leq 6\sigma_{max}^4(2+\epsilon_1)n\log p\;\;.$$
\end{alemma}
\begin{proof}Fix some edge $\mathbf{e}=\{j,k\}$. Then
$$(s_n(\Theta_0))_{(j,k)}=\frac{n}{2}(\Sigma_0)_{jk}-\frac{1}{2}X_j^TX_k=-\frac{1}{2}\sum_{i=1}^n((X_j)_i(X_k)_i-(\Sigma_0)_{jk})\;\;.$$
Write $Y_j=((\Sigma_0)_{jj})^{-1}X_j$, $Y_k=((\Sigma_0)_{kk})^{-1}X_k$, $\rho=((\Sigma_0)_{jj}(\Sigma_0)_{kk})^{-1}(\Sigma_0)_{jk}=\mathrm{corr}(Y_j,Y_k)$. Then
$$(s_n(\Theta_0))_{(j,k)}=-\frac{1}{2}(\Sigma_0)_{jj}(\Sigma_0)_{kk}\sum_{i=1}^n((Y_j)_i(Y_k)_i-\rho)\;\;.$$
By Lemma~\ref{lem:2}, there are some independent $A,B\sim\chi^2_n$ such that
$$(s_n(\Theta_0))_{(j,k)}=-\frac{1}{2}(\Sigma_0)_{jj}(\Sigma_0)_{kk}\left[\left(\frac{1+\rho}{2}\right)(A-n)-\left(\frac{1-\rho}{2}\right)(B-n)\right]\;\;.$$
There are ${p\choose 2}\leq \frac{1}{2}p^2$ edges in the complete graph. Therefore, by the union bound, it will suffice to show that, with probability at least $1-(\frac{1}{2}p^2)^{-1}\frac{1}{\sqrt{\pi \log p}}e^{-\epsilon_1\log p}$,
$$\frac{1}{4}\sigma_{max}^4\left[\left(\frac{1+\rho}{2}\right)(A-n)-\left(\frac{1-\rho}{2}\right)(B-n)\right]^2\leq 6\sigma_{max}^4(2+\epsilon_1)n\log p\;\;.$$
Suppose this bound does not hold. Then
$$\left|\left(\frac{1+\rho}{2}\right)(A-n)\right|>\sqrt{6(2+\epsilon_1)n\log p} \ \ \text{or} \ \ \left|\left(\frac{1-\rho}{2}\right)(B-n)\right|>\sqrt{6(2+\epsilon_1)n\log p}\;\;.$$
Since $\rho\in[-1,1]$, this implies that
$$\left|A-n\right|>\sqrt{6(2+\epsilon_1)n\log p} \ \ \text{or} \ \ \left|B-n\right|>\sqrt{6(2+\epsilon_1)n\log p}\;\;.$$
Since $A\stackrel{\mathcal{D}}{=}B$, it will suffice to show that with probability at least $1-p^{-2}\frac{1}{\sqrt{\pi \log p}}e^{-\epsilon_1\log p}$,
$$\left|A-n\right|\leq \sqrt{6(2+\epsilon_1)n\log p}\;\;.$$
Write $\lambda=\sqrt{6(2+\epsilon_1)\frac{\log p}{n}}$. Observe that, by assumption (\ref{eqn_Assumption1}), $\lambda\leq \frac{1}{2}$ and $n\geq 3$; therefore (by Taylor series),
$$\frac{n}{2}(\lambda-\log(1+\lambda))\geq\frac{n}{2}\left( \frac{\lambda^2}{2}-\frac{\lambda^3}{3}\right)\geq\frac{n}{2}\cdot\frac{\lambda^2}{3}=(2+\epsilon_1)\log p\;\;, \ \text{and}$$
$$-\frac{n-1}{2}(\lambda+\log(1-\lambda))\geq\frac{n-1}{2}\left( \frac{\lambda^2}{2}\right)\geq\frac{n}{2}\cdot\frac{\lambda^2}{3}=(2+\epsilon_1)\log p\;\;.$$
Furthermore,
$$\lambda\sqrt{n-1}=\sqrt{6(2+\epsilon_1)\log p\times\frac{n-1}{n}}\geq\sqrt{\log p}\;\;.$$
By (CSB) from the paper,
$$P\{A-n>\sqrt{6(2+\epsilon_1)n\log p}\}=P\{A>n(1+\lambda)\}\leq\frac{1}{\lambda\sqrt{\pi n}}e^{-\frac{n}{2}(\lambda-\log(1+\lambda))}$$
$$\leq \frac{1}{\lambda\sqrt{\pi(n-1)}}e^{-(2+\epsilon_1)\log p}\leq \frac{1}{\sqrt{\pi\log p}}e^{-(2+\epsilon_1)\log p}\;\;,$$
and also,
$$P\{A-n<-\sqrt{6(2+\epsilon_1)n\log p}\}=P\{A<n(1-\lambda)\}\leq\frac{1}{\lambda\sqrt{\pi (n-1)}}e^{\frac{n-1}{2}(\lambda+\log(1-\lambda))}$$
$$\leq \frac{1}{\lambda\sqrt{\pi(n-1)}}e^{-(2+\epsilon_1)\log p}\leq \frac{1}{\sqrt{\pi\log p}}e^{-(2+\epsilon_1)\log p}\;\;.$$
This gives the desired result.\end{proof}

\begin{alemma}\label{likelihoodbound} Recall that, in the proof of Theorem~\ref{thm:1-nonasy}, we showed that
$$l_n(\Theta)-l_n(\Theta_0)\leq\sqrt{\theta_0^2(p+2q)\times 6\s^4_{max}(2+\e_1)n\log p}-\frac{1}{2}\theta_0^2\times \frac{n}{2}(2\lambda_{max})^{-2}\;\;.$$
Then this implies that
$$l_n(\Theta)-l_n(\Theta_0)\leq-2q(\log p)(1+\gamma_0)\;\;.$$
\end{alemma}
\begin{proof}
It is sufficient to show that
$$\sqrt{\theta_0^2(p+2q)\times 6\s^4_{max}(2+\e_1)n\log p}-\frac{1}{2}\theta_0^2\times \frac{n}{2}(2\lambda_{max})^{-2}\leq -(p+2q)(\log p)(1+\gamma_0)\;\;.$$
We rewrite this as
$$\sqrt{A\times n^2\theta_0^4\lambda_{max}^{-2} 6\s^4_{max}(2+\e_1)}-\frac{1}{2}\theta_0^2\times \frac{n}{2}(2\lambda_{max})^{-2}\leq -A\times n\times\theta_0^2\lambda_{max}^{-2}(1+\gamma_0)\;\;,$$
where
$$A=\frac{(p+2q)\log p}{n}\times\frac{\lambda_{max}^2}{\theta_0^2}\;\;.$$
Using $C\geq \sigma_{max}^2\lambda_{max}$, it's sufficient to show that
$$\sqrt{A\times n^2\theta_0^4\lambda_{max}^{-4} 6C^2(2+\e_1)}-\frac{1}{2}\theta_0^2\times \frac{n}{2}(2\lambda_{max})^{-2}\leq -A\times n\times\theta_0^2\lambda_{max}^{-2}(1+\gamma_0)\;\;.$$
Dividing out common factors, the above is equivalent to showing that
$$\sqrt{A\times  6C^2(2+\e_1)}-\frac{1}{16}\leq -A\times(1+\gamma_0)\;\;.$$
By assumption (\ref{eqn_Assumption1}), we know:
$$A\times(1+\gamma_0)\leq\frac{1}{3200}\;\;,$$
and also,
$$A\times 6C^2(2+\epsilon_1)\leq 12\times\frac{1}{3200}\;\;.$$
Therefore,
$$A\times(1+\gamma_0)+\sqrt{A\times 6C^2(2+\epsilon_1)}\leq\frac{1}{3200}+\sqrt{\frac{12}{3200}}<\frac{1}{16}\;\;,$$
as desired.
\end{proof}

Lemma~\ref{likelihoodbound2} below is sufficient to fill in the details of Theorem~\ref{thm:2-nonasy}.
\begin{alemma}\label{likelihoodbound2} Recall that, in the proof of Theorem~\ref{thm:2-nonasy}, we showed that, stochastically,
 $$l_n(\hat{\Theta}(\ss))-l_n(\hat{\Theta}(\so))\leq \frac{n}{2}\times\frac{1}{n-\sqrt{2q}-1}\chi^2_m\;\;.$$
Then this implies that
$$P\{l_n(\hat{\Theta}(\ss))-l_n(\hat{\Theta}(\so))\geq 2(1+\gamma_0)m\log(p)\}\leq \frac{1}{4\sqrt{\pi}\log p}e^{-\frac{m}{2}(4(1+\frac{\e_0}{2})\log p)}\;\;.$$
\end{alemma}
\begin{proof}
First, we show that $\frac{n-\sqrt{2q}-1}{n}\geq(1+\gamma_0)^{-\frac{1}{2}}$. By assumption (\ref{eqn_Assumption2}) we see that:
$$\frac{1}{\log p}\leq 4(\sqrt{1+\gamma_0}-1)\;\;.$$
Now turn to assumption~(\ref{eqn_Assumption1}). We see that the right-hand side of~(\ref{eqn_Assumption1}) is $\leq \frac{1}{4\sqrt{1+\gamma_0}}$. On the left-hand side of~(\ref{eqn_Assumption1}), by definition, $\lambda_{max}^2\geq \theta_0^2$. Therefore,
$$\frac{(p+2q)\log p}{n}\leq \frac{1}{4\sqrt{1+\gamma_0}}\;\;.$$
Therefore,
$$\frac{\sqrt{2q}+1}{n}\leq\frac{p+2q}{n}\leq \frac{4(\sqrt{1+\gamma_0}-1)}{4\sqrt{1+\gamma_0}}=1-\frac{1}{\sqrt{1+\gamma_0}}\;\;,$$
and so,
$$\frac{n-\sqrt{2q}-1}{n}\geq (1+\gamma_0)^{-\frac{1}{2}}\;\;.$$
Therefore, using the stochastic inequality in the statement in the lemma,
\begin{eqnarray*}
P\{l_n(\hat{\Theta}(\ss))-l_n(\hat{\Theta}(\so))&\geq& 2(1+\gamma_0)m\log(p)\}\\
&\leq& P\{\chi^2_m\geq 4(1+\gamma_0)m\log p \times \frac{n-\sqrt{2q}-1}{n}\}\\
&\leq& P\{\chi^2_m\geq 4\sqrt{1+\gamma_0}m\log p\}\;\;.\\
\end{eqnarray*}
Now we apply the chi-square bound from~\cite{cai:2002}, and obtain that
$$ P\{\chi^2_m\geq 4\sqrt{1+\gamma_0}m\log p\}\leq \frac{1}{(4\sqrt{1+\gamma_0}\log p-1)\sqrt{\pi m}}e^{-\frac{m}{2}(4\sqrt{1+\gamma_0}\log p-1-\log(4\sqrt{1+\gamma_0}\log p))}\;\;.$$
Since $m\geq 1$ and $\frac{1}{\log p}\leq 4(\sqrt{1+\gamma_0}-1)$, we obtain that the upper bound is at most
\begin{eqnarray*}
&&\frac{1}{4\sqrt{\pi}\log p}e^{-\frac{m}{2}(4\sqrt{1+\gamma_0}\log p-1-\log(4\sqrt{1+\gamma_0}\log p))}\\
&=&\frac{1}{4\sqrt{\pi}\log p}e^{-\frac{m}{2}(4\sqrt{1+\gamma_0}\log p-(\log\log p +\log(4\sqrt{1+\gamma_0})+1))}\\
&=&\frac{1}{4\sqrt{\pi}\log p}e^{-\frac{m}{2}(2\log p)(2\sqrt{1+\gamma_0}-(\log\log p +\log(4\sqrt{1+\gamma_0})+1)/(2\log p))}\;\;.\\
\end{eqnarray*}
By assumption (\ref{eqn_Assumption2}), we may further bound this expression from above as 
$$\frac{1}{4\sqrt{\pi}\log p}e^{-\frac{m}{2}(2\log p)(2+\epsilon_0)}=\frac{1}{4\sqrt{\pi}\log p}e^{-\frac{m}{2}4(1+\frac{\epsilon_0}{2})\log p}\;\;.$$
\end{proof}
\end{appendix}

\end{document}